\documentclass[a4paper]{amsart}
\usepackage{color}
\usepackage{amssymb,amsmath,amsthm,amstext,amsfonts}
\usepackage[dvips]{graphicx}
\usepackage{psfrag}
\usepackage{url}
\usepackage{amsfonts}
\usepackage{amsmath}
\usepackage{mathrsfs}
\usepackage{graphicx}
\usepackage{amsmath,amsthm,amssymb,amscd}
\usepackage{enumerate}
\usepackage[colorlinks,linkcolor=black,anchorcolor=black,citecolor=black,hyperindex=true,CJKbookmarks=true]{hyperref}

\usepackage{epsfig}
\usepackage[T1]{fontenc}
\usepackage{setspace}
\makeatletter \@addtoreset{equation}{section} \makeatother

\renewcommand\thetable{\thesection.\@arabic\c@table}

\theoremstyle{plain}
\newtheorem{maintheorem}{Theorem}

\newtheorem{maincorollary}{Corollary}

\newtheorem{Thm}{Theorem}[section]
\newtheorem{Lem}[Thm]{Lemma}
\newtheorem{Prop}[Thm]{Proposition}

\newtheorem{Que}[Thm]{Question}
\theoremstyle{remark}
\newtheorem{Def}[Thm] {Definition}
\newtheorem{Rem}[Thm] {Remark}

\makeatletter
\@namedef{subjclassname@2020}{2020 Mathematics Subject Classification}
\makeatother

\begin{document}

\title{Non-uniform cocycles for some uniquely ergodic minimal dynamical systems on connected spaces}

\author{Wanshan Lin and Xueting Tian}

\address{Wanshan Lin, School of Mathematical Sciences,  Fudan University\\Shanghai 200433, People's Republic of China}
\email{21110180014@m.fudan.edu.cn}

\address{Xueting Tian, School of Mathematical Sciences,  Fudan University\\Shanghai 200433, People's Republic of China}
\email{xuetingtian@fudan.edu.cn}

\begin{abstract}
In this paper, we pay attention to a weaker version of Walters's question on the existence of non-uniform cocycles for uniquely ergodic minimal dynamical systems on non-degenerate connected spaces. We will classify such dynamical systems into three classes: not totally uniquely ergodic; totally uniquely ergodic but not topological weakly mixing; totally uniquely ergodic and topological weakly mixing. We will give an affirmative answer to such question for the first two classes. Also, we will show the existence of such dynamical systems in the first class with arbitrary topological entropy. 
\end{abstract}

\maketitle
\section{Introduction}
 Let $\mathbb{Z}$, $\mathbb{N}$, $\mathbb{N}^+$ denote integers, non-negative integers, positive integers, respectively. Throughout this paper, we always use the pair $(X,f)$ to denote a \emph{dynamical system}, which means that $(X,d)$ is a compact metric space and $f:X \rightarrow X$ is a \emph{homeomorphism}. Let $\mathfrak{B}(X)$ denote the Borel $\sigma$-algebra of $X$. Let $\mathcal{M}(X)$, $\mathcal{M}_f(X)$, $\mathcal{M}_f^{e}(X)$ denote the space of Borel probability measures, $f$-invariant Borel probability measures, $f$-ergodic Borel probability measures, respectively. Let $C(X,\mathbb{C})$ (resp. $C(X,\mathbb{R})$) denote the space of real (resp. complex) continuous functions on $X$ with the norm $\|\varphi\|:=\sup\limits_{x\in X}|\varphi(x)|.$ Let $|A|$ denote the cardinality of the set $A$. 
 
 A dynamical system $(X,f)$ is said to be \emph{uniquely ergodic} if $|\mathcal{M}_f(X)|=1$, it's equivalent to $|\mathcal{M}^e_f(X)|=1$. Sometimes, to emphasis $\mu$ is the unique ergodic measure, we will say that $(X,\mu,f)$ is uniquely ergodic. It's known that for a given uniquely ergodic system $(X,\mu,f)$ and $\varphi\in C(X,\mathbb{R})$, we always have that $\frac{1}{n}S_n\varphi(x)$ converges uniformly to $\int\varphi\mathrm{d}\mu$ on $X$, where $S_n\varphi(x):=\sum_{i=0}^{n-1}\varphi(f^ix)$. A sequence $(\varphi_n)_{n\geq1}$ of $C(X,\mathbb{R})$ is called \emph{subadditive} if for every $x\in X$ and $n,m\geq 1$, the equality $\varphi_{n+m}(x)\leq\varphi_n(f^mx)+\varphi_m(x)$ is satisfied. It's clear that $(S_n\varphi)_{n\geq1}$ is subadditive whenever $\varphi\in C(X,\mathbb{R})$. When $(X,\mu,f)$ is uniquely ergodic and $\mu$ is non-atomic, it's shown in \cite[Example 6.3]{Derriennic-Krengel-1981}, there always exists a subadditive sequence $(\phi_n)_{n\geq1}$ of $C(X,\mathbb{R})$ such that $\frac{1}{n}\phi_n(x)$ does not converge uniformly on $X$. Let $C(X,\mathrm{GL}(d))$ denote the space of all continuous functions from $X$ to $\mathrm{GL}(d)$, where $\mathrm{GL}(d)$ is the set of invertible $d\times d$ matrices with real entries. Given $A\in C(X,\mathrm{GL}(d))$ and $n\in\mathbb{Z}$, the \emph{cocycle} $A(n,x)$ generated by $A$ is defined as 
 \[A(n,x)=\begin{cases}
 	A(f^{n-1}x)A(f^{n-2}x)\cdots A(x) &\text{if }n>0,\\
 	I_d &\text{if }n=0,\\
 	A^{-1}(f^nx)A^{-1}(f^{n+1}x)\cdots A^{-1}(f^{-1}x) &\text{if }n<0.
 \end{cases}
 \]
 Then it can be checked that $(\log\|A(n,x)\|)_{n\geq1}$ is subadditive. A cocycle $A$ is \emph{uniform} if $\frac{1}{n}\log\|A(n,x)\|$ converges uniformly on $X$, it's clear the convergence is independent with the choice of matrix norm. A map $A\in C(X,\mathrm{GL}(d))$ is uniform if the cocycle generated by $A$ is uniform. The following question was asked by Walters \cite{Walters-1986}.
 \begin{Que}\cite{Walters-1986}
 	When $(X,\mu,f)$ is uniquely ergodic and $\mu$ is non-atomic, does there exist a non-uniform $A\in C(X,\mathrm{GL}(2))$?
 \end{Que}
 
 Also, in \cite{Walters-1986}, Walters proved that if $(X,f)$ and $(X,f^2)$ is minimal, $(X,\mu,f)$ is uniquely ergodic while $(X,\mu,f^2)$ is not uniquely ergodic, then there exists a non-uniform $A\in C(X,\mathrm{GL}(2))$. The method to construct such a dynamical system was shown by Veech in \cite{Veech-1969}. For suitable irrational rotations on the circle, the existence of non-uniform map in $C(X,\mathrm{GL}(2))$ was shown by Herman \cite{Herman-1981}. The result was generalized by Lenz \cite{Lenz-2004} to every irrational rotation on the circle, by using Lenz's result, we can give an affirmative answer to Walters's question when $X$ is the circle (Proposition \ref{Prop 2.5}). In \cite{Furman-1997}, Furman studied the sufficient conditions for $A\in C(X,\mathrm{GL}(d))$ to be uniform when $(X,\mu,f)$ is uniquely ergodic, and the necessary conditions for $A\in C(X,\mathrm{GL}(2))$ to be uniform if further $(X,f)$ is minimal. The latter result was also generalized by Lenz \cite{Lenz-2004} without the assumption that $(X,f)$ is minimal.  
    
 Given a measure-preserving system $(X,\mu,f)$ and $A\in C(X,\mathrm{GL}(d))$, we define $\Lambda_{\mu}(A)$ as $$\Lambda_{\mu}(A):=\lim_{n\to\infty}\frac{1}{n}\int\log\|A(n,x)\|\mathrm{d}\mu(x)
 =\inf_{n\geq1}\frac{1}{n}\int\log\|A(n,x)\|\mathrm{d}\mu(x).$$
 Then $\Lambda_{\mu}(A)$ is independent with the choice of matrix norm and by Kingman's subadditive ergodic theorem, $$\lim_{n\to\infty}\frac{1}{n}\log\|A(n,x)\|=\Lambda_{\mu}(A)\text{ for }\mu\text{-}a.e.\text{ }x\in X.$$
 When $(X,\mu,f)$ is uniquely ergodic, given $A\in C(X,\mathrm{GL}(d))$, it's shown by Furman \cite[Corollary 2]{Furman-1997}, for every $x\in X$ and uniformly on $X$, $$\limsup_{n\to\infty}\frac{1}{n}\log\|A(n,x)\|\leq\Lambda_{\mu}(A),$$ if further $A$ is uniform, then by Proposition \ref{Prop-2.2}, $\frac{1}{n}\log\|A(n,x)\|$ converges uniformly to $\Lambda_{\mu}(A)$ on $X$.
 
 Recall that $(X,f)$ is said to be \emph{minimal} if for every $x\in X$, its orbit $\mathrm{orb}(x,f):=\{f^ix:i\in\mathbb{N}\}$ is dense in $X$; \emph{topological transitive} if for every pair of non-empty open subsets $U$ and $V$ of $X$, there exists $k\in\mathbb{N}^+$ such that $f^k(U)\cap V\neq\emptyset$; \emph{totally transitive} if $(X,f^n)$ is topological transitive for any $n\in\mathbb{N}^+$; \emph{topological weakly mixing} if $(X\times X,T\times T)$ is topological transitive; \emph{topological mixing} if for every pair of non-empty open subsets $U$ and $V$ of $X$, there exists $N\in\mathbb{N}^+$ such that $f^k(U)\cap V\neq\emptyset$ for any $k\geq N$. A uniquely ergodic dynamical system $(X,\mu,f)$ is said to be \emph{totally uniquely ergodic} if $(X,\mu,f^n)$ is uniquely ergodic for any $n\in\mathbb{N}^+$.
 
 In this paper, we will consider a weaker version of Walters's question for uniquely ergodic minimal dynamical systems on non-degenerate (i.e. with at least two points) connected spaces.
 \begin{Que}
 	Suppose that $(X,f)$ is a uniquely ergodic minimal dynamical system on a non-degenerate connected space $X$, does there exist a non-uniform $A\in C(X,\mathrm{GL}(m))$ for some $m\geq2$?
 \end{Que}
 Uniquely ergodic minimal dynamical system on connected spaces are always totally transitive \cite[Theorem 2.5]{Banks-1997} and can be classified into one of the following three classes:
 \begin{enumerate}[(I).]
 	\item Uniquely ergodic but not totally uniquely ergodic, minimal dynamical systems on connected spaces;
 	\item Totally uniquely ergodic, minimal but not topological weakly mixing dynamical systems on connected spaces;
 	\item Totally uniquely ergodic, minimal and topological weakly mixing dynamical systems on connected spaces.
 \end{enumerate}
 The well known examples in class (II) is the irrational rotations on the circle. An example in class (III) can be found in \cite{Pavlov-2008}. As for examples in class (I) will be shown in Theorem \ref{Theorem F}.
 \begin{maintheorem}\label{Theorem A}
 	Suppose that $(X,f)$ is a uniquely ergodic minimal dynamical system on a connected space $X$ and in class (I) and (II), then there exists a non-uniform $A\in C(X,\mathrm{GL}(m))$ for some $m\geq2$.
 \end{maintheorem}
 
 It's still unknown whether we can have the same result for dynamical systems in class (III). Theorem \ref{Theorem A} is directly deduced from Theorem \ref{Theorem B}, Theorem \ref{Theorem C} and Theorem \ref{Theorem D}, which will be introduced in the following. 
 
 \begin{maintheorem}\label{Theorem B}
 	Suppose that $(X,f)$ is uniquely ergodic, minimal, totally transitive but not topological weakly mixing, then for any $m\geq2$, there exists a non-uniform $A\in C(X,\mathrm{GL}(m))$.
 \end{maintheorem}
 \begin{Rem}
 	The proof of Theorem \ref{Theorem B} is based on Lenz's work on the existence of non-uniform cocycles for irrational rotations on the circle. One of the keys is the basic fact (Proposition \ref{Proposition 2.1}): the extension of a dynamical system having non-uniform cocycles also has non-uniform cocycles.
 \end{Rem}
 Given $x\in X$ and $n\in\mathbb{N}^+$, the $n$-order \emph{empirical measure} of $x$ for $f$ is denoted by $$\mathcal{E}_{f,n}(x):=\frac{1}{n}\sum_{i=0}^{n-1}\delta_{f^ix}.$$
 Let $V_f(x)$ denote the set of accumulation points of $\{\mathcal{E}_{f,n}(x):n\in\mathbb{N}^+\}$, then it's a non-empty compact connected subset of $\mathcal{M}_f(X)$ (\cite[Proposition 3.8]{Denker-Grillenberger-Sigmund-1976}). Given $\mu\in\mathcal{M}_f(X)$, denote $$G_{f,\mu}:=\left\{x\in X:V_f(x)=\{\mu\}\right\}.$$  
 Every point $x\in G_{f,\mu}$ is said to be a \emph{generic} point of $\mu$ for $f$. When $\mu\in\mathcal{M}_f^e(X)$, $\mu(G_{f,\mu})=1$ (\cite[Proposition 5.9]{Denker-Grillenberger-Sigmund-1976}). Denote $$Q_f^e(X):=\bigcup_{\mu\in\mathcal{M}_f^e(X)}G_{f,\mu}.$$
 
 \begin{maintheorem}\label{Theorem C}
 	Suppose that $(X,f)$ is uniquely ergodic and there exists a prime $p\in\mathbb{N}^+$ such that $Q_{f^p}^e(X)\neq X$, then for any $m\geq p$, there exists a non-uniform $A\in C(X,\mathrm{GL}(m))$.
 \end{maintheorem}
 \begin{Rem}
 	By Lemma \ref{Lemma 3.1}, we know that if $(X,f)$ and $(X,f^2)$ are minimal, $(X,\mu,f)$ is uniquely ergodic while $(X,\mu,f^2)$ is not uniquely ergodic, then $Q_{f^2}^e(X)\neq X$. Hence, Theorem \ref{Theorem C} is a generalization of Walters's work on the existence of non-uniform cocycles for such dynamical systems.  
 \end{Rem}
  We will give a sufficient condition for a dynamical system satisfying the assumptions in Theorem \ref{Theorem C}.
  \begin{maintheorem}\label{Theorem D}
  	Suppose that $(X,f)$ is totally transitive, uniquely ergodic but not totally uniquely ergodic, then there exists a prime $p\in\mathbb{N}^+$  such that $Q_{f^p}^e(X)\neq X$.
  \end{maintheorem}

  Next, we will give examples to show the existences of topological mixing dynamical systems satisfying the assumptions in Theorem \ref{Theorem D}. Given $\mu\in\mathcal{M}_f(X)$, $(X,\mu,f)$ is said to be \emph{weakly mixing} if $(X\times X,\mu\times\mu,f\times f)$ is ergodic. Let $h_\mu(f)$ denote the metric entropy for any $\mu\in\mathcal{M}_f(X)$, we have
   \begin{maintheorem}\label{Theorem E} Suppose that $(X,f)$ is a dynamical system and $\mu\in\mathcal{M}_f(X)$ is a weakly mixing measure, then for any prime $p$, there exists a dynamical system $(Y,g)$, such that 
  \begin{enumerate}[(1)]
  	\item $(Y,\nu,g^i)$ is uniquely ergodic for any $1\leq i\leq p-1$, but $(Y,\nu,g^p)$ is not uniquely ergodic;
  	
  	\item $(X,\mathfrak{B}(X), \mu,f)$ is a factor of $(Y,\mathfrak{B}(Y), \nu,g)$ and $h_\nu(g)=h_\mu(f)$;
  	
  	\item $(Y,g)$ is topological mixing.
  	
  \end{enumerate}
  \end{maintheorem}
  As a corollary, we have 
  \begin{maincorollary}\label{Corollary A}
  	Given $\alpha\geq0$ and a prime $p$, there exists a dynamical system $(Y,g)$, such that
  	\begin{enumerate}[(1)]
  		\item $(Y,\nu,g^i)$ is uniquely ergodic for any $1\leq i\leq p-1$, but $(Y,\nu,g^p)$ is not uniquely ergodic;
  		\item $(Y,g)$ is topological mixing;
  		\item $h_\nu(g)=h_{top}(g)=\alpha$, where $h_{top}(g)$ is the topological entropy of $(Y,g)$.
  	\end{enumerate}
  \end{maincorollary}
 Based on Corollary \ref{Corollary A}, we will shown the existences of dynamical systems in class (I), more precisely, let $\mathbb{T}^2$ denote the 2-torus, we have
 \begin{maintheorem}\label{Theorem F}
 	Given $\alpha\geq0$ and a prime $p$, there exists a minimal dynamical system $(\mathbb{T}^2,g)$, such that 
 	\begin{enumerate}[(1)]
 		\item $(\mathbb{T}^2,\nu,g^i)$ is uniquely ergodic for any $1\leq i\leq p-1$, but $(\mathbb{T}^2,\nu,g^p)$ is not uniquely ergodic;
 		\item $h_\nu(g)=h_{top}(g)=\alpha$, where $h_{top}(g)$ is the topological entropy of $(\mathbb{T}^2,g)$.
 	\end{enumerate}
 \end{maintheorem}
  
\textbf{Organization of this paper.} In section \ref{section 2}, we introduce some preliminary results. In section \ref{section 3}, we prove Theorem \ref{Theorem B}. In section \ref{section 4}, we prove Theorem \ref{Theorem C} and Theorem \ref{Theorem D}. In section 5, we prove Theorem \ref{Theorem E}, Corollary \ref{Corollary A} and Theorem \ref{Theorem F}.

\section{Preliminaries}\label{section 2}
\subsection{Extensions}
Given two dynamical systems $(X,f)$ and $(Y,g)$, if there exists a continuous surjection $\pi:X\to Y$ with $\pi\circ f=g\circ\pi$, then we say that $\pi$ is a \emph{factor map}, $(Y,g)$ is a \emph{factor} of $(X,f)$ or $(X,f)$ is an \emph{extension} of $(Y,g)$. If further, $\pi$ is a homeomorphism, we say that $(X,f)$ and $(Y,g)$ is \emph{topological conjugate}. It's know that the factor of a topological transitive dynamical system is also topological transitive. 
\begin{Prop}\label{Proposition 2.1}
	Suppose that $(X,f)$ is an \emph{extension} of $(Y,g)$ and there exists a non-uniform $A\in C(Y,\mathrm{GL}(d))$, then there exists a non-uniform $B\in C(X,\mathrm{GL}(d))$.
\end{Prop}
\begin{proof}
	Let $\pi: X\to Y$ be the factor map and $B=A\circ\pi$, then for any $n\geq 1$, $$B(n,x)=B(f^{n-1}x)B(f^{n-2}x)\cdots B(x)=A(\pi(f^{n-1}x))A(\pi(f^{n-2}x))\cdots A(\pi x)=A(n,\pi x).$$ Since $\pi$ is a surjection and $A$ is non-uniform, we have that $B$ is non-uniform.
\end{proof}
\subsection{Eigenfunctions and eigenvalues}
Given $\varphi\in C(X,\mathbb{C})\setminus\{0\}$, $\varphi$ is said to be an \emph{eigenfunction} for $(X,f)$ if there exists $\lambda\in\mathbb{C}$ such that $\varphi(fx)=\lambda\varphi(x)$ for any $x\in X$, such $\lambda$ is said to be the \emph{eigenvalue} for $(X,f)$ corresponding to the eigenfunction $\varphi$.
\begin{Lem}\cite[Theorem 5.17]{Walters-1982}\label{Lem2.2}
	Suppose that $(X,f)$ is topological transitive, $\varphi$ is a eigenfunction for $(X,f)$ and $\lambda$ is the eigenvalue corresponding to $\varphi$, then $|\lambda|=1$ and $|\varphi|$ is constant.
\end{Lem}
$(X,f)$ is said to have \emph{topological discrete spectrum} if the smallest closed linear subspace of $C(X,\mathbb{C})$ containing the eigenfunctions of $(X,f)$ is $C(X,\mathbb{C})$.
\subsection{Equicontinuous dynamical systems}
$(X,f)$ is said to be \emph{equicontinuous}, if for any $\varepsilon>0$, there exists $\delta>0$ such that $d(x,y)<\delta$ implies that $d(f^nx,f^ny)<\varepsilon$ for any $n\in\mathbb{Z}$. If $(X,f)$ is equicontinuous, then $(X,f^n)$ is equicontinuous for any $n\in\mathbb{Z}$.
\begin{Lem}\label{Lem2.3}\cite{Walters-1982,Huang-Shao-Ye-2023}(Halmos-von Neumann Theorem) Suppose that $(X,f)$ is a topological transitive equicontinuous dynamical system, then it is topological conjugate to a minimal rotation on a compact abelian metric group and has topological discrete spectrum.
\end{Lem}
\subsection{Maximal equicontinuous factors}
$(Y,g)$ is said to be a \emph{equicontinuous factor} of $(X,f)$ if $(Y,g)$ is equicontinuous and $(Y,g)$ is a factor of $(X,f)$. Suppose that $(Y,g)$ is a equicontinuous factor of $(X,f)$, it is said to be maximal, if every equicontinuous factor of $(X,f)$ is a equicontinuous factor of $(Y,g)$. Let $(X_{eq}, f_{eq})$ denote the maximal equicontinuous factors of $(X,f)$.
\begin{Lem}\cite[Theorem 2.12]{Huang-Shao-Ye-2023}\label{Lem2.4}
	Suppose that $(X,f)$ is a minimal dynamical system, then $(X,f)$ is topological weakly mixing if and only if $(X_{eq}, f_{eq})$ is trivial (i.e. $|X_{eq}|>1$).
\end{Lem}

\subsection{Some basic facts for invariant measures}
Given $\mu,\nu\in\mathcal{M}_f^e(X)$ with $\mu\neq\nu$, then $0\leq\mu(G_{f,\nu})\leq\mu(X\setminus G_{f,\mu})=0$, hence, $\mu(G_{f,\nu})=0$.
\begin{Lem}\label{Lemma 2.1}
	Suppose that $\mu_1,\mu_2,\cdots,\mu_m,\nu_1,\nu_2,\cdots,\nu_s\in\mathcal{M}_f^e(X)$, if $\mu_1\notin\{\nu_1,\nu_2,\cdots,\nu_s\}$, then for any $p_1,p_2,\cdots,p_m>0$ with $\sum_{i=1}^{m}p_i=1$ and $q_1,q_2,\cdots,q_s>0$ with $\sum_{j=1}^{s}q_j=1$, denote $\mu=\sum_{i=1}^{m}p_i\mu_i$ and $\nu=\sum_{j=1}^{s}q_j\nu_j$, we have that $\mu\neq\nu$.
\end{Lem}
\begin{proof}
	Since $\mu(G_{f,\mu_1})=p_1>0=\nu(G_{f,\mu_1})$, we have that $\mu\neq\nu$.
\end{proof}
\subsection{Some basic facts for uniform cocycles} In this subsection, we introduce some basic facts for uniform cocycles.
\begin{Lem}\cite[Lemma 2.2]{Walters-1986}\label{Lemma 2.2}
	Suppose that $(X,\mu,f)$ is uniquely ergodic, if $A\in C(X,\mathrm{GL}(d))$ is uniform, then $\frac{1}{n}\log\|A(n,x)\|$ converges uniformly to a constant on $X$.
\end{Lem}
Combing with Kingman's subadditive ergodic theorem, we have that
\begin{Prop}\label{Prop-2.2}
	Suppose that $(X,\mu,f)$ is uniquely ergodic, then $A\in C(X,\mathrm{GL}(d))$ is uniform if and only if $\frac{1}{n}\log\|A(n,x)\|$ converges uniformly to $\Lambda_{\mu}(A)$ on $X$. 
\end{Prop}

\begin{Prop}\label{Proposition 2.7}
	Suppose that $(X,\mu,f)$ is uniquely ergodic, if there exist $m\in\mathbb{N}^+$ and a non-uniform $A\in C(X,\mathrm{GL}(m))$, then for any $p\geq m$, there exists a non-uniform $B\in C(X,\mathrm{GL}(p))$.
\end{Prop}
\begin{proof}
	Let $\tilde{A}=\frac{1}{|\det(A)|^{\frac{1}{m}}}A$, then $\tilde{A}\in C(X,\mathrm{GL}(m))$ with $|\det(\tilde{A})|\equiv1$ and $$\frac{1}{n}\log\|\tilde{A}(n,x)\|=\frac{1}{n}\log\|A(n,x)\|-\frac{1}{m}\frac{1}{n}\sum_{i=0}^{n-1}\log|\det(A(f^ix))|.$$ Hence, $\tilde{A}$ is non-uniform. Let $\|\cdot\|_\infty$ denote the $\infty$-norm. Since $|\det(\tilde{A})|\equiv1$, we have that $|\det(\tilde{A}(n,x))|=1$ for any $n\geq 1$ and $x\in X$. As a result, $\|\tilde{A}(n,x)\|_\infty\geq1$ for any $n\geq 1$ and $x\in X$. Given $p\geq m+1$, denote $B=\mathrm{diag}\left\{\tilde{A},I_{p-m}\right\}$, then $B\in C(X,\mathrm{GL}(p))$ and for any $n\geq 1$ and $x\in X$, $$\|B(n,x)\|_\infty=\max\left\{\|\tilde{A}(n,x)\|_\infty,1\right\}=\|\tilde{A}(n,x)\|_\infty.$$
	Therefore, $B$ is non-uniform.
\end{proof}

\subsection{$f$-ergodic but not $f^d$-ergodic}\label{section 2.7} When $(X,\mu,f)$ is uniquely ergodic and $(X,\mu,f^d)$ is not uniquely ergodic for some $d\geq2$, $\mathcal{M}_{f^d}^e(X)$ can be generated by an element in it. More precisely, when $d=2$, the following was shown by Walters.
\begin{Lem}\cite[Lemma 2.6]{Walters-1986}
	Suppose that $(X,\mu,f)$ is uniquely ergodic and $(X,\mu,f^2)$ is not uniquely ergodic, then there exists $\nu\in\mathcal{M}_{f^2}^e(X)$ such that $\mathcal{M}_{f^2}^e(X)=\{\nu,\nu\circ f^{-1}\}$ and $\mu=\frac{1}{2}(\nu+\nu\circ f^{-1})$.
\end{Lem}
We can generalize the result to any $d\geq2$. 
\begin{Lem}\label{Lem 2.8}
	Suppose that $(X,\mu,f)$ is uniquely ergodic and $(X,\mu,f^d)$ is not uniquely ergodic for some $d\geq2$, then there exists $\nu\in\mathcal{M}_{f^d}^e(X)$, such that $\mathcal{M}_{f^d}^e(X)=\{\nu\circ f^{-i}: 0\leq i\leq d-1\}$ and $\mu=\frac{1}{d}\sum_{i=0}^{d-1}\nu\circ f^{-i}$.
\end{Lem}
\begin{proof}
	Choose $\nu\in\mathcal{M}_{f^d}^e(X)$, then $\nu\circ f^{-d}=\nu$ and for any $1\leq k\leq d-1$, $\nu\circ f^{-k}\in\mathcal{M}_{f^d}^e(X)$, since $(X,\mu,f)$ is uniquely ergodic, we have that 
	\begin{equation}
		\mu=\frac{1}{d}(\nu+\nu\circ f^{-1}+\cdots+\nu\circ f^{-(d-1)}).
	\end{equation} 
	If there exists $\tilde{\nu}\in\mathcal{M}_{f^d}^e(X)\setminus\{\nu\circ f^{-i}: 0\leq i\leq d-1\}$, then similarly, 
	\begin{equation}
		\mu=\frac{1}{d}(\tilde{\nu}+\tilde{\nu}\circ f^{-1}+\cdots+\tilde{\nu}\circ f^{-(d-1)}).
	\end{equation} 
	It's a contradiction to Lemma \ref{Lemma 2.1}. Hence, $\mathcal{M}_{f^d}^e(X)=\{\nu\circ f^{-i}: 0\leq i\leq d-1\}$.
\end{proof}

\begin{Lem}\label{Lemma 2.12}
	Suppose that $(X,\mu,f)$ is uniquely ergodic, $p$ and $q$ are primes such that $(X,\mu,f^{pq})$ is not uniquely ergodic, then either $(X,\mu,f^p)$ is not uniquely ergodic or $(X,\mu,f^q)$ is not uniquely ergodic.
\end{Lem}
\begin{proof}
	Without loss of generality, we can assume that $(X,\mu,f^q)$ is uniquely ergodic. Similarly, by Lemma \ref{Lem 2.8}, there exists $\nu\in\mathcal{M}_{f^{pq}}^e(X)$, $\mathcal{M}_{f^{pq}}^e(X)=\{\nu\circ f^{-i}: 0\leq i\leq pq-1\}$ and $\mu=\frac{1}{pq}\sum_{i=0}^{pq-1}\nu\circ f^{-i}$. Let $m=\min\{1\leq i\leq pq:\nu\circ f^{-i}=\nu\}$, then $m\geq 2$ and $\nu\circ f^{-i}\neq\nu\circ f^{-j}$ for any $0\leq i<j\leq m-1$. Since $\nu\circ f^{-pq}=\nu$, we have that $m$ divides $pq$, hence, $m=p$, $m=q$ or $m=pq$. 
	
	If $m=p$, $\nu,\nu\circ f^{-1}\in\mathcal{M}_{f^p}(X)$, hence, $(X,\mu,f^p)$ is not uniquely ergodic. 
	
	If $m=q$, $\nu,\nu\circ f^{-1}\in\mathcal{M}_{f^q}(X)$, hence, $(X,\mu,f^q)$ is not uniquely ergodic, it's a contradiction to the assumptions.
	
	If $m=pq$, for $i=0,1$, let $\mu_i=\frac{1}{q}\sum_{j=0}^{q-1}\nu\circ f^{-jp-i}$. Then $\mu_1,\mu_2\in\mathcal{M}_{f^p}(X)$ and by Lemma \ref{Lemma 2.1}, we have that $\mu_1\neq\mu_2$. Hence, $(X,\mu,f^p)$ is not uniquely ergodic.
\end{proof}
\begin{Prop}\label{Proposition 2.13}
	Suppose that $(X,\mu,f)$ is uniquely ergodic and $(X,\mu,f^d)$ is not uniquely ergodic for some $d\geq2$, then there exists a prime $p$ dividing $d$ such that $(X,\mu,f^p)$ is not uniquely ergodic.
\end{Prop}
\begin{proof}
	Suppose that $d=p_1p_2\cdots p_k$, where $p_i$ are all primes, we proceed by induction on $k$. When $k=1$, the conclusion is clear. When $k=2$, the conclusion is from Lemma \ref{Lemma 2.12}. Now, given $m\geq2$, we suppose that the conclusion holds provided $k\leq m$. Then for $k=m+1$, denote $g=f^{p_1p_2\cdots p_{m-1}}$, then $f^{p_1p_2\cdots p_k}=f^{p_1p_2\cdots p_{m+1}}=g^{p_mp_{m+1}}$.
	
	If $(X,\mu,g)$ is not uniquely ergodic, then by the inductive hypothesis, there exists $1\leq i\leq m-1$ such that $(X,\mu,f^{p_i})$ is not uniquely ergodic.
	
	If $(X,\mu,g)$ is uniquely ergodic, from Lemma \ref{Lemma 2.12}, either $(X,\mu,g^{p_m})$ is not uniquely ergodic or $(X,\mu,g^{p_{m+1}})$ is not uniquely ergodic. Since $g^{p_m}=f^{p_1p_2\cdots p_{m-1}p_m}$ and $g^{p_{m+1}}=f^{p_1p_2\cdots p_{m-1}p_{m+1}}$, by the inductive hypothesis, there exists $1\leq i\leq m+1$ such that $(X,\mu,f^{p_i})$ is not uniquely ergodic.
	
	As a result, the conclusion holds provided $k\leq m+1$. By induction, we complete this proof.
	
\end{proof}
\begin{Lem}\label{Lemma 2.14}
	Suppose that $(X,\mu,f)$ is uniquely ergodic and $(X,\mu,f^p)$ is not uniquely ergodic for some prime $p$, then there exists $\nu\in\mathcal{M}_{f^p}^e(X)$ such that $\nu\circ f^{-i}\neq\nu\circ f^{-j}$ for any $0\leq i<j\leq p-1$, $\mathcal{M}_{f^p}^e(X)=\{\nu\circ f^{-i}: 0\leq i\leq p-1\}$ and $\mu=\frac{1}{p}\sum_{i=0}^{p-1}\nu\circ f^{-i}$.
\end{Lem}

\begin{proof}
    Since $(X,\mu,f^p)$ is not uniquely ergodic, by Lemma \ref{Lem 2.8}, there exists $\nu\in\mathcal{M}_{f^p}^e(X)$, $\mathcal{M}_{f^p}^e(X)=\{\nu\circ f^{-i}: 0\leq i\leq p-1\}$ and $\mu=\frac{1}{p}\sum_{i=0}^{p-1}\nu\circ f^{-i}$. Let $m=\min\{1\leq i\leq p:\nu\circ f^{-i}=\nu\}$, then $m\geq 2$ and $\nu\circ f^{-i}\neq\nu\circ f^{-j}$ for any $0\leq i<j\leq m-1$. Since $\nu\circ f^{-p}=\nu$, we have that $m$ divides $p$, hence, $m=p$.
\end{proof}

\section{Proof of Theorem \ref{Theorem B}}\label{section 3}
Let $\mathbb{T}^1:=\{z\in\mathbb{C}:|z|=1\}$ denote the circle. Given an irrational number $\alpha\in(0,1)$, the \emph{irrational rotation} $R_\alpha$ on $\mathbb{T}^1$ is defined by $R_\alpha(z)=e^{2\pi i\alpha}z$.

The proof of Theorem \ref{Theorem B} will based on Lenz's result for irrational rotations on the circle. 
\begin{Lem}\cite[Remark 3]{Lenz-2004}\label{Lemm3.1}
	There exists a non-uniform $A\in C(\mathbb{T}^1,\mathrm{GL}(2))$ for every irrational rotation on $\mathbb{T}^1$.
\end{Lem}
Now, we give the proof to Theorem \ref{Theorem B}.
\subsection{Proof of Theorem \ref{Theorem B}} By Lemma \ref{Lem2.4}, $(X_{eq},f_{eq})$ is non-trivial. Since $(X,f)$ is transitive, we have that $(X_{eq},f_{eq})$ is transitive, by Lemma \ref{Lem2.3}, $(X_{eq},f_{eq})$ has topological discrete spectrum. Since $(X_{eq},f_{eq})$ is non-trivial, there exists a non-constant eigenfunction $\varphi$ for $(X_{eq},f_{eq})$ and an eigenvalue $\lambda$ corresponding to $\varphi$, by Lemma \ref{Lem2.2}, $|\varphi|$ is constant and $|\lambda|=1$. If $\lambda^{k_0}=1$ for some $k_0\in\mathbb{N}^+$, then $\varphi(f^{k_0}x)=\varphi(x)$ for any $x\in X_{eq}$. Since $(X_{eq},f^{k_0}_{eq})$ is a factor of $(X,f^{k_0})$ and $(X,f^{k_0})$ is topological transitive, we have that $(X_{eq},f_{eq}^{k_0})$ is topological transitive, combining with the fact that $\varphi$ is continuous, we have that $\varphi$ is constant, it's a contradiction. As a result, $\lambda^k\neq 1$ for any $k\in\mathbb{N}^+$. Hence, there exists an irrational number $\alpha\in(0,1)$ with $\lambda=e^{2\pi i\alpha}$. Define $$\pi:X_{eq}\to\mathbb{T}^1\text{ by } \pi=\frac{\varphi}{|\varphi|},$$ then $\pi$ is a continuous surjection and $\pi\circ f_{eq}=R_\alpha\circ\pi$. Therefore, $(\mathbb{T}^1,R_\alpha)$ is a factor of $(X_{eq},f_{eq})$ and thus is a factor of $(X,f)$. By Lemma \ref{Lemm3.1}, there exists a non-uniform $B_1\in C(\mathbb{T}^1,\mathrm{GL}(2))$ for $(\mathbb{T}^1,R_\alpha)$. By Proposition \ref{Proposition 2.1}, there exists a non-uniform $B_2\in C(X,\mathrm{GL}(2))$ for $(X,f)$. By Proposition \ref{Proposition 2.7}, for any $m\geq2$, there exists a non-uniform $A\in C(X,\mathrm{GL}(m))$.\qed

From the proof, we have the following proposition.
\begin{Prop}
	Suppose that $(X,f)$ is an extension of an irrational rotation on $\mathbb{T}^1$, then for any $m\geq2$, there exists a non-uniform $A\in C(X,\mathrm{GL}(m))$.
\end{Prop}

When $X$ is $\mathbb{T}^1$, we can get more information. The following is a classic result for homeomorphisms on $\mathbb{T}^1$.
\begin{Lem}\cite[Theorem 6.18]{Walters-1982}
	Suppose that $f$ is a homeomorphism on $\mathbb{T}^1$ without periodic points, then there exists an irrational rotation $(\mathbb{T}^1,R_\alpha)$ such that $(\mathbb{T}^1,f)$ is an extension of $(\mathbb{T}^1,R_\alpha)$.
\end{Lem}
Combining Proposition \ref{Proposition 2.1} and Lemma \ref{Lemm3.1}, we can give an affirmative answer to Walters's question when $X=\mathbb{T}^1$. 
\begin{Prop}\label{Prop 2.5}
	Suppose that $f$ is a homeomorphism on $\mathbb{T}^1$, $(\mathbb{T}^1,\mu,f)$ is uniquely ergodic and $\mu$ is non-atomic, then for any $m\geq2$, there exists a non-uniform $A\in C(\mathbb{T}^1,\mathrm{GL}(m))$.
\end{Prop}

\section{Proofs of Theorem \ref{Theorem C} and Theorem \ref{Theorem D} }\label{section 4}
\begin{Lem}\label{Lemma 2.8}
	Suppose that $(X,f)$ is a dynamical system with $2\leq|\mathcal{M}_f^e(X)|<\infty$, then for any $\nu\in\mathcal{M}_f^e(X)$, there exists $\varphi\in C(X,\mathbb{R})$ such that $\int\varphi\mathrm{d}\nu>\int\varphi\mathrm{d}\mu$ for any $\mu\in\mathcal{M}_f(X)$ with $\mu\neq\nu$.
\end{Lem}
\begin{proof}
	Suppose that $\mathcal{M}_f^e(X)=\{\nu,\mu_1,\mu_2,\cdots,\mu_{n-1}\}$, then for any $1\leq i\leq n-1$, we have that $\mu_i(G_{f,\nu})=0$, by regularity, there exists an open subset $U_i\supset G_{f,\nu}$ with $\mu_i(U_i)\leq\frac{1}{5}$. Denote $F=X\setminus\bigcap_{i=1}^{n-1}U_i$, then $F$ is closed and $\mu_i(F)\geq\frac{4}{5}$ for any $1\leq i\leq n-1$. Since $\nu(G_{f,\nu})=1$, by regularity, we can find a closed subset $E\subset G_{f,\nu}$ with $\nu(E)\geq\frac{4}{5}$. We define $\varphi\in C(X,\mathbb{R})$ by $$\varphi(x)=\frac{d(x,F)}{d(x,F)+d(x,E)},$$ then $\int\varphi\mathrm{d}\nu\geq\frac{4}{5}$ and $\int\varphi\mathrm{d}\mu_i\leq\frac{1}{5}$ for any $1\leq i\leq n-1$. Given $\mu\in\mathcal{M}_f(X)$ with $\mu\neq\nu$, then there exist $t_0\in [0,1)$ and $t_1,\cdots,t_{n-1}\in [0,1]$ with $\sum_{i=0}^{n-1}t_i=1$ and $\mu=t_0\nu+\sum_{i=1}^{n-1}t_i\mu_i$. As a result, $\int\varphi\mathrm{d}\nu>\int\varphi\mathrm{d}\mu$.
\end{proof}

Now, we give the proof to Theorem \ref{Theorem C}.
\subsection{Proof of Theorem \ref{Theorem C}} 
Since $Q_{f^p}^e(X)\neq X$, $(X,\mu,f^p)$ is not uniquely ergodic. By Lemma \ref{Lemma 2.14}, there exists $\nu\in\mathcal{M}_{f^p}^e(X)$ such that $\nu\circ f^{-i}\neq\nu\circ f^{-j}$ for any $0\leq i<j\leq p-1$, $\mathcal{M}_{f^p}^e(X)=\{\nu\circ f^{-i}: 0\leq i\leq p-1\}$ and $\mu=\frac{1}{p}\sum_{i=0}^{p-1}\nu\circ f^{-i}$. For $(X,f^p)$, by Lemma \ref{Lemma 2.8}, there exists $\varphi\in C(X,\mathbb{R})$ such that $\int\varphi\mathrm{d}\nu>\int \varphi\mathrm{d}\mu$ for any $\mu\in\mathcal{M}_{f^p}(X)$ with $\mu\neq\nu$. We define $A\in C(X,\mathrm{GL}(p))$ by 
\[A(x)=\begin{pmatrix}
	0 & I_{p-1}\\
	e^{\varphi(x)} & 0\\
\end{pmatrix}.
\]
Then $$A(p,x)=\mathrm{diag}\left\{e^{\varphi(x)}, e^{\varphi(fx)},\cdots, e^{\varphi(f^{p-1}x)}\right\}.$$ 
For any $n\in\mathbb{N}^+$, 
\begin{align*}
	A(np,x)&=A(p,f^{(n-1)p}x)A(p,f^{(n-2)p}x)\cdots A(p,x)\\
	&=\mathrm{diag}\left\{e^{\sum_{i=0}^{n-1}\varphi(f^{ip}x)}, e^{\sum_{i=0}^{n-1}\varphi(f^{ip+1}x)},\cdots, e^{\sum_{i=0}^{n-1}\varphi(f^{ip+p-1}x)}\right\}.
\end{align*}

Now, suppose that $A$ is uniform. 

Let $\|\cdot\|_\infty$ denote the $\infty$-norm, then $$\frac{1}{np}\log\|A(np,x)\|_\infty=\frac{1}{p}\max\left\{\frac{1}{n}\sum_{i=0}^{n-1}\varphi(f^{ip}x),\frac{1}{n}\sum_{i=0}^{n-1}\varphi(f^{ip+1}x),\cdots, \frac{1}{n}\sum_{i=0}^{n-1}\varphi(f^{ip+p-1}x)\right\}.$$ Since $(X,\mu,f)$ is uniquely ergodic and $\frac{1}{n}\log\|A(n,x)\|_\infty$ converges uniformly on $X$, by Lemma \ref{Lemma 2.2}, $\frac{1}{n}\log\|A(n,x)\|_\infty$ converges uniformly to a constant on $X$. In particular, $\frac{1}{np}\log\|A(np,x)\|_\infty$ converges uniformly to a constant on $X$. Choose $x_0\in G_{f^p,\nu}$, then $f^ix_0\in G_{f^p,\nu\circ f^{-i}}$ for any $1\leq i\leq p-1$. Hence, $$\lim_{n\to\infty}\frac{1}{np}\log\|A(np,x_0)\|_\infty=\frac{1}{p}\max\left\{\int\varphi\mathrm{d}\nu,\int\varphi\mathrm{d}(\nu\circ f^{-1}),\cdots,\int\varphi\mathrm{d}(\nu\circ f^{-(p-1)})\right\}=\frac{1}{p}\int\varphi\mathrm{d}\nu.$$
As a result, $\frac{1}{np}\log\|A(np,x)\|_\infty$ converges uniformly to $\frac{1}{p}\int\varphi\mathrm{d}\nu$ on $X$.

Given $x\in X$ and $\mu\in V_{f^p}(x)$, there exist $n_1<n_2<\cdots$ such that $\mu=\lim_{j\to\infty}\frac{1}{n_j}\sum_{i=0}^{n_j-1}\delta_{f^{ip}x}$. Hence, for any $1\leq k\leq p-1$, we have that $\mu\circ f^{-k}=\lim_{j\to\infty}\frac{1}{n_j}\sum_{i=0}^{n_j-1}\delta_{f^{ip+k}x}$. As a result, $$\frac{1}{p}\int\varphi\mathrm{d}\nu=\lim_{j\to\infty}\frac{1}{n_jp}\log\|A(n_jp,x)\|_\infty=\frac{1}{p}\max\left\{\int\varphi\mathrm{d}\mu,\int\varphi\mathrm{d}(\mu\circ f^{-1}),\cdots,\int\varphi\mathrm{d}(\mu\circ f^{-(p-1)})\right\}.$$ 
Hence, $\nu\in\left\{\mu,\mu\circ f^{-1},\cdots,\mu\circ f^{-(p-1)}\right\}$ and we have that $\mu\in\mathcal{M}_{f^p}^e(X)$. Recall that $V_{f^p}(x)$ is a non-empty compact connected subset of $\mathcal{M}_{f^p}(X)$, we have that $|V_{f^p}(x)|=1$. As a result, $x\in Q_{f^p}^e(X)$. Therefore, $Q_{f^p}^e(X)=X$, it's a contradiction to that $Q_{f^p}^e(X)\neq X$. As a result, $A$ is non-uniform. By Proposition \ref{Proposition 2.7}, for any $m\geq p$, there exists a non-uniform $A\in C(X,\mathrm{GL}(m))$.\qed

Given $\mu\in\mathcal{M}_f(X)$, denote $$G^{f,\mu}:=\{x\in X:V_f(x)\supset\{\mu\}\},$$ then $G_{f,\mu}\subset G^{f,\mu}$. Let $\rho$ be a metrization of the weak* topology of $\mathcal{M}(X)$. Then $$G^{f,\mu}=\bigcap_{k=1}^{\infty}\bigcup_{n\geq k}\left\{x\in X:\rho(\mathcal{E}_{f,n}(x),\mu)<\frac{1}{k}\right\}.$$ Hence, $G^{f,\mu}$ is a $G_\delta$ subset of $X$.
\begin{Lem}\label{Lemma 3.1}
	Suppose that $(X,\mu,f)$ is uniquely ergodic, $(X,\mu,f^p)$ is not uniquely ergodic for some prime $p$ and $(X,f^p)$ is topological transitive, then $Q^e_{f^p}(X)\neq X$.
\end{Lem}
\begin{proof}
	Suppose that $Q^e_{f^p}(X)=X$, since $(X,f^p)$ is topological transitive, there exists $x_0\in X$ such that $orb(x_0,f^p):=\{f^{pi}x_0:i\in\mathbb{N}\}$ is dense in $X$. Since $(X,\mu,f^p)$ is not uniquely ergodic, by Lemma \ref{Lemma 2.14}, there exists $\nu\in\mathcal{M}_{f^p}^e(X)$ such that $\nu\circ f^{-i}\neq\nu\circ f^{-j}$ for any $0\leq i<j\leq p-1$ and $\mathcal{M}_{f^p}^e(X)=\{\nu\circ f^{-i}: 0\leq i\leq p-1\}$. Hence, $$X=\bigcup_{i=0}^{p-1}G_{f^p,\nu\circ f^{-i}}=\bigcup_{i=0}^{p-1}f^i(G_{f^p,\nu}).$$ As a result, there exists $0\leq j_0\leq p-1$ such that $x_0\in f^{j_0}(G_{f^p,\nu})$. Since $f^{j_0}(G_{f^p,\nu})$ is $f^p$-invariant, we have that $f^{j_0}(G_{f^p,\nu})$ is dense in $X$. Since $f$ is a homeomorphism, we have that $f^i(G_{f^p,\nu})$ is dense in $X$ for any $0\leq i\leq p-1$, which means that $G_{f^p,\nu\circ f^{-i}}$ is dense in $X$ for any $0\leq i\leq p-1$. Since $G_{f^p,\nu\circ f^{-i}}\subset G^{f^p,\nu\circ f^{-i}}$, we have that $G^{f^p,\nu\circ f^{-i}}$ is a dense $G_\delta$ subset of $X$ for any $0\leq i\leq p-1$. In particular, $G^{f^p,\nu}\cap G^{f^p,\nu\circ f^{-1}}\neq\emptyset$. Hence, there exists $y_0\in X$ with $V_{f^p}(y_0)\supset\{\nu,\nu\circ f^{-1}\}$, it's a contradiction to $Q^e_{f^p}(X)=X$. Therefore, we have that $Q^e_{f^p}(X)\neq X$. 
\end{proof}
\subsection{Proof of Theorem \ref{Theorem D}} Since $(X,\mu,f)$ is not totally uniquely ergodic, there exists $d\geq2$ such that $(X,\mu,f^d)$ is not uniquely ergodic. By Proposition \ref{Proposition 2.13}, there exists a prime $p$ dividing $d$ such that $(X,\mu,f^p)$ is not uniquely ergodic. Since $(X,f)$ is totally transitive, we have that $(X,f^p)$ is topological transitive, by Lemma \ref{Lemma 3.1}, we have that $Q^e_{f^p}(X)\neq X$. By Theorem \ref{Theorem C}, we have that for any $m\geq p$, there exists a non-uniform $A\in C(X,\mathrm{GL}(m))$.\qed

\section{Proofs of Theorem \ref{Theorem E}, Corollary \ref{Corollary A} and Theorem \ref{Theorem F}}\label{section 5}
We introduce some preliminary results before we prove Theorem \ref{Theorem E} and Corollary \ref{Corollary A}. 
\subsection{Isomorphism of probability spaces and measure-preserving transformations}
\subsubsection{Isomorphism of probability spaces}
Two probability spaces $(X_1,\mathfrak{A}_1,\mu_1)$ and $(X_2,\mathfrak{A}_2,\mu_2)$ is \emph{isomorphic} if there exist $Y_1\in\mathfrak{A}_1$, $Y_2\in\mathfrak{A}_2$ with $\mu_1(Y_1)=\mu_2(Y_2)=1$ and an invertible measure-preserving transformation $\phi:Y_1\to Y_2$, where the space $Y_i$ is assumed to be equipped with the $\sigma$-algebra $Y_i\cap\mathfrak{A}_i:=\{Y_i\cap A:A\in\mathfrak{A}_i\}$ for $i=1,2$.
\subsubsection{Lebesgue space}
We use the classic definition of Lebesgue space. A probability space $(X,\mathfrak{A},\mu)$ is a \emph{Lebesgue space} if isomorphic to $([0,1],\mathfrak{B}([0,1]),l)$, where $l$ is the Lebesgue measure on closed unit interval $[0,1]$. The following is a classic result for Lebesgue space.
\begin{Lem}\cite[Theorem 2.1]{Walters-1982}\label{Lemma 5.1}
	Suppose that $X$ is a complete separable metric space and $\mu$ is a non-atomic Borel probability measure, then $(X,\mathfrak{B}(X),\mu)$ is a Lebesgue space.
\end{Lem}
\subsubsection{Isomorphism of measure-preserving transformations}
\begin{Def}\label{Definition 5.2}
	A measure-preserving transformation $(X_1,\mathfrak{A}_1,\mu_1,f_1)$ is said to be \emph{a factor of } a measure-preserving transformation $(X_2,\mathfrak{A}_2,\mu_2,f_2)$ if there exists $Y_1\in\mathfrak{A}_1$, $Y_2\in\mathfrak{A}_2$ with $\mu_1(Y_1)=\mu_2(Y_2)=1$, $f_1(Y_1)\subset Y_1$, $f_2(Y_2)\subset Y_2$ and an measure-preserving map $\phi:Y_1\to Y_2$, where the space $Y_i$ is assumed to be equipped with the $\sigma$-algebra $Y_i\cap\mathfrak{A}_i:=\{Y_i\cap A:A\in\mathfrak{A}_i\}$ for $i=1,2$, such that $\phi\circ f_1(x)=f_2\circ\phi(x)$ for any $x\in Y_1$.
\end{Def}
\begin{Def}\label{Definition 5.3}
	Two measure-preserving transformations $(X_1,\mathfrak{A}_1,\mu_1,f_1)$ and $(X_2,\mathfrak{A}_2,\mu_2,f_2)$ is said to be \emph{isomorphic} if there exists $Y_1\in\mathfrak{A}_1$, $Y_2\in\mathfrak{A}_2$ with $\mu_1(Y_1)=\mu_2(Y_2)=1$, $f_1(Y_1)\subset Y_1$, $f_2(Y_2)\subset Y_2$ and an invertible measure-preserving map $\phi:Y_1\to Y_2$, where the space $Y_i$ is assumed to be equipped with the $\sigma$-algebra $Y_i\cap\mathfrak{A}_i:=\{Y_i\cap A:A\in\mathfrak{A}_i\}$ for $i=1,2$, such that $\phi\circ f_1(x)=f_2\circ\phi(x)$ for any $x\in Y_1$.
\end{Def}

When $(X_1,\mathfrak{A}_1,\mu_1,f_1)$ is isomorphic to $(X_2,\mathfrak{A}_2,\mu_2,f_2)$, it's known that 
\begin{enumerate}[(1)]
	\item $(X_1,\mathfrak{A}_1,\mu_1,f_1)$ is ergodic if and only if $(X_2,\mathfrak{A}_2,\mu_2,f_2)$ is ergodic;
	\item $(X_1,\mathfrak{A}_1,\mu_1,f_1^n)$ is isomorphic to $(X_2,\mathfrak{A}_2,\mu_2,f_2^n)$ for any $n\geq1$.
\end{enumerate}
\begin{Lem}\cite[Theorem 4.11]{Walters-1982}\label{Lemma 5.4}
	Suppose that $(X_1,\mathfrak{A}_1,\mu_1,f_1)$ is isomorphic to $(X_2,\mathfrak{A}_2,\mu_2,f_2)$, then $h_{\mu_1}(f_1)=h_{\mu_2}(f_2)$.
\end{Lem}

\begin{Lem}\cite[Theorem]{Lehrer-1987}\label{Lemma 5.5}
	Suppose that $(X,\mathfrak{B}(X),\mu,f)$ is an invertible ergodic measure preserving transformation of a Lebesgue space $(X,\mathfrak{B}(X),\mu)$, then $(X,\mathfrak{B}(X),\mu,f)$ is isomorphic to a $(Y,\mathfrak{B}(Y),\nu,g)$ such that $(Y,g)$ is a topological mixing dynamical system and $(Y,\nu,g)$ is uniquely ergodic.
\end{Lem}

\subsection{Weakly mixing measures}
\begin{Lem}\cite[Corollary 9.21]{Eisner-Farkas-Haase-Nagel-2015}\label{Lemma 5.6}
	Suppose that $(X,\mu,f)$ is weakly mixing, then $(X,\mu,f^n)$ is weakly mixing for any $n\in\mathbb{N}^+$.
\end{Lem}
\begin{Lem}\cite[Theorem 9.23]{Eisner-Farkas-Haase-Nagel-2015}\label{Lemma 5.7}
	Suppose that $(X,\mu,f)$ is weakly mixing and $(Y,\nu,g)$ is ergodic, then $(X\times Y,\mu\times\nu,f\times g)$ is ergodic.
\end{Lem}
\begin{Lem}\label{Lemma 5.8}
	There exists a dynamical system $(X,f)$ and $\mu\in\mathcal{M}_f(X)$ such that $(X,\mu,f)$ is weakly mixing, $\mu$ is non-atomic and $h_\mu(f)=0$.
\end{Lem}
\begin{proof}
	Let $(X,f)$ be a hyperbolic automorphism of $2$-torus, then it is Axiom $A$ with only one basic set $\Omega_s$, $(\Omega_s,f|_{\Omega_s})$ is topological mixing and the set of weakly mixing invariant measures is residual in $\mathcal{M}_f(X)$ \cite{Sigmund-1972}. Furthermore, the set of non-atomic invariant measures is residual in $\mathcal{M}_f(X)$ and the set of invariant measures with zero metric entropy is residual in $\mathcal{M}_f(X)$ \cite{Sigmund-1970}. As a result, there exists $\mu\in\mathcal{M}_f(X)$ such that $(X,\mu,f)$ is weakly mixing, $\mu$ is non-atomic and $h_\mu(f)=0$.
\end{proof}
\begin{Lem}\label{Lemma 5.9}
	Given $\alpha\geq0$, there exist a dynamical system $(X,f)$ and $\mu\in\mathcal{M}_f(X)$ such that $(X,\mu,f)$ is weakly mixing and $h_\mu(f)=\alpha$.
\end{Lem}
\begin{proof}
	When $\alpha=0$, it's directly from Lemma \ref{Lemma 5.8}.
	
	When $\alpha>0$, by \cite[Remark, Page 105]{Walters-1982}, there is a Bernoulli shift $(X,\mu,f)$ with $h_{\mu}(f)=\alpha$, and thus it's a Kolmogorov automorphism \cite[Theorem 4.30]{Walters-1982}, hence, it's weakly mixing \cite[Corollary 4.33.1]{Walters-1982}.
\end{proof}

Now, we begin the proofs of Theorem \ref{Theorem E} and Corollary \ref{Corollary A}.
\subsection{Proof of Theorem \ref{Theorem E}}
Given a weakly mixing measure $\mu\in\mathcal{M}_f(X)$ and a prime $p$, let $X_1:=\{x,f_1x,f_1^2x,\cdots,f_1^{p-1}x\}$, where $f_1:X_1\to X_1$ is a homeomorphism and $x$ is a periodic point with minimum period $p$. Denote $\mu_1:=\frac{1}{p}\sum_{i=0}^{p-1}\delta_{f_1^ix}$, then $h_{\mu_1}(f_1)=0$, $(X_1,\mu_1,f_1^i)$ is ergodic for any $1\leq i\leq p-1$ and $(X_1,\mu_1,f_1^p)$ is not ergodic. Let $(X_2,\mu_2,f_2)$ be a dynamical system from Lemma \ref{Lemma 5.8}, denote $Y_1=X_1\times X_2, \nu_1=\mu_1\times\mu_2, g_1=f_1\times f_2$, then $\nu_1$ is non-atomic and by \cite[Theorem 4.23]{Walters-1982}, $h_{\nu_1}(g_1)=h_{\mu_1}(f_1)+h_{\mu_2}(f_2)=0$. And by Lemma \ref{Lemma 5.6} and Lemma \ref{Lemma 5.7}, $(Y_1, \nu_1,g_1^i)$ is ergodic for any $1\leq i\leq p-1$ and $(Y_1,\nu_1,g_1^p)$ is not ergodic.  Denote $Y_2=X\times Y_1$, $\nu_2=\mu\times\nu_1$, $g_2=f\times g_1$, then $\nu_2$ is non-atomic and by \cite[Theorem 4.23]{Walters-1982}, $h_{\nu_2}(g_2)=h_{\mu}(f)+h_{\nu_1}(g_1)=h_{\mu}(f)$. And by Lemma \ref{Lemma 5.6} and Lemma \ref{Lemma 5.7}, $(Y_2, \nu_2,g_2^i)$ is ergodic for any $1\leq i\leq p-1$ and $(Y_2,\nu_2,g_2^p)$ is not ergodic. Furthermore, $(X,\mathfrak{B}(X),\mu,f)$ is a factor of $(Y_2,\mathfrak{B}(Y_2)\nu_2,g_2)$. By Lemma \ref{Lemma 5.1}, $(Y_2,\mathfrak{B}(Y_2),\nu_2)$ is a Lebesgue space and thus by Lemma \ref{Lemma 5.5}, $(Y_2,\mathfrak{B}(Y_2),\nu_2,g_2)$ is isomorphic to a $(Y,\mathfrak{B}(Y),\nu,g)$ such that $(Y,g)$ is a topological mixing dynamical system and $(Y,\nu,g)$ is uniquely ergodic. Hence, $(Y,\nu,g^p)$ is not ergodic and $(Y,\nu,g^i)$ is ergodic for any $1\leq i\leq p-1$, combining with Lemma \ref{Lemma 2.14}, we have that $(Y,\nu,g^i)$ is uniquely ergodic for any $1\leq i\leq p-1$ and $(Y,\nu,g^p)$ is not uniquely ergodic. By Lemma \ref{Lemma 5.4}, $h_{\nu}(g)=h_{\nu_2}(g_2)=h_{\mu}(f)$. Since $(X,\mathfrak{B}(X),\mu,f)$ is a factor of $(Y_2,\mathfrak{B}(Y_2)\nu_2,g_2)$, we have that $(X,\mathfrak{B}(X), \mu,f)$ is a factor of $(Y,\mathfrak{B}(Y), \nu,g)$.\qed
\subsection{Proof of Corollary \ref{Corollary A}}
Given $\alpha\geq0$, by Lemma \ref{Lemma 5.9}, there exist a dynamical system $(X,f)$ and $\mu\in\mathcal{M}_f(X)$ such that $(X,\mu,f)$ is weakly mixing with $h_{\mu}(f)=\alpha$. Given a prime $p$, by Theorem \ref{Theorem E}, there exists a dynamical system $(Y,g)$, such that \begin{enumerate}[(i)]
	\item $(Y,\nu,g^i)$ is uniquely ergodic for any $1\leq i\leq p-1$, but $(Y,\nu,g^p)$ is not uniquely ergodic;
	
	\item $(X,\mathfrak{B}(X), \mu,f)$ is a factor of $(Y,\mathfrak{B}(Y), \nu,g)$ and $h_\nu(g)=h_\mu(f)=\alpha$;
	
	\item $(Y,g)$ is topological mixing.
\end{enumerate}
	
Since $(Y,\nu,g)$ is uniquely ergodic, by the variational principle \cite[Theorem 8.6]{Walters-1982}, we have that $h_{top}(g)=h_\nu(g)=\alpha$.\qed

Before we prove Theorem \ref{Theorem F}, we need some lemmas, let $\lambda_{\mathbb{T}^1}$ denote the normalized Lebesgue measure on $\mathbb{T}^1$.
\begin{Lem}\cite[Pages 254 and 256]{Beguin-Crovisier-Le Roux-2007}\label{Lemma 5.10}
	Suppose that $(X,\mathfrak{B}(X),\mu,f)$ is a ergodic measure-preserving transformation and there exists an irrational number $\alpha\in(0,1)$ such that $(\mathbb{T}^1, \mathfrak{B}(\mathbb{T}^1),\lambda_{\mathbb{T}^1},R_\alpha)$ is a factor of $(X,\mathfrak{B}(X),\mu,f)$. Then there exists a minimal dynamical system $(\mathbb{T}^2,g)$ such that $(\mathbb{T}^2,\nu,g)$ is uniquely ergodic and $(\mathbb{T}^2,\mathfrak{B}(\mathbb{T}^2),\nu,g)$ is isomorphism to $(X,\mathfrak{B}(X),\mu,f)$.
\end{Lem}
\begin{Lem}\cite{Konieczny-Kupsa-Kwietniak-2018}\label{Lemma 5.11}
	Given a prime $p$ and a dynamical system $(X,f)$ such that $(X,\mu,f^i)$ is ergodic for any $1\leq i\leq p-1$, but $(X,\mu,f^p)$ is not ergodic. Then there exists an irrational number $\alpha\in(0,1)$ such that $(X\times\mathbb{T}^1, \mathfrak{B}(X\times\mathbb{T}^1),\mu\times\lambda_{\mathbb{T}^1},f^i\times R_\alpha)$ is ergodic for any $1\leq i\leq p-1$, but $(X\times\mathbb{T}^1, \mathfrak{B}(X\times\mathbb{T}^1),\mu\times\lambda_{\mathbb{T}^1},f^p\times R_\alpha)$ is not ergodic. 
\end{Lem}
\begin{proof}
	This is from the proof of \cite[Corollary 17]{Konieczny-Kupsa-Kwietniak-2018}, the key of that proof is using the three facts: the point spectrum of a measure-preserving transformation is at most countable \cite[Page 3428]{Konieczny-Kupsa-Kwietniak-2018}; the product of two ergodic measure-preserving transformations is ergodic if and only if their point spectra have trivial intersection \cite[Lemma 1]{Konieczny-Kupsa-Kwietniak-2018}; the set of irrational numbers in $(0,1)$ are uncountable. 
\end{proof}
\subsection{Proof of Theorem \ref{Theorem F}}
Given $\alpha\geq0$ and a prime $p$, by Corollary \ref{Corollary A}, there exists a dynamical system $(X,f)$ such that $(X,\mu,f^i)$ is uniquely ergodic for any $1\leq i\leq p-1$, $(X,\mu,f^p)$ is not uniquely ergodic and $h_\mu(f)=\alpha$. By Lemma \ref{Lemma 5.11}, there exists an irrational number $\beta\in(0,1)$ such that $(X\times\mathbb{T}^1, \mathfrak{B}(X\times\mathbb{T}^1),\mu\times\lambda_{\mathbb{T}^1},f^i\times R_\beta)$ is ergodic for any $1\leq i\leq p-1$, but $(X\times\mathbb{T}^1, \mathfrak{B}(X\times\mathbb{T}^1),\mu\times\lambda_{\mathbb{T}^1},f^p\times R_\beta)$ is not ergodic. Let $\nu=\mu\times\lambda_{\mathbb{T}^1}$, by Lemma \ref{Lemma 5.7}, $h_{\nu}(f\times R_\beta)=h_\mu(f)=\alpha$. Combine with Lemma \ref{Lemma 5.10}, we finish the proof.\qed

\bigskip

$\mathbf{Acknowledgements.}$ The authors are partially supported by the National Natural
Science Foundation of China (No. 12071082)  and  Natural Science Foundation of Shanghai (No. 23ZR1405800).

\end{document}